\documentclass{birkjour}
\usepackage{amssymb,verbatim,enumerate,ifthen}
\usepackage[mathscr]{eucal}
\usepackage[utf8]{inputenc}
\usepackage[T1]{fontenc}
\newtheorem{thm}{Theorem}[section]
\newtheorem{Athm}{Theorem}

\newtheorem{cor}[thm]{Corollary}

\newtheorem{lem}[thm]{Lemma}

\theoremstyle{definition}

\numberwithin{equation}{section}
\def\eq#1{{\rm(\ref{#1})}}
\def\Eq#1#2{\ifthenelse{\equal{#1}{*}}
  {\begin{equation*}\begin{aligned}[]#2\end{aligned}\end{equation*}}
  {\begin{equation}\begin{aligned}[]\label{#1}#2\end{aligned}\end{equation}}}

\def\A{\mathscr{A}}
\def\B{\mathscr{B}}
\def\D{\mathscr{D}}

\def\G{\mathscr{G}}
\def\H{\mathscr{H}}

\def\calC{\mathcal{C}}
\newcommand{\operator}[1]{\mathop{\vphantom{\sum}\mathchoice
{\vcenter{\hbox{\LARGE $#1$}}}
{\vcenter{\hbox{\large $#1$}}}{#1}{#1}}\displaylimits}

\def\Mst_#1^#2{\operator{\mathscr{M}_{\mbox{\scriptsize$\#$}}\!\!}_{#1}^{#2}\,\,}

\newcommand\R{\mathbb{R}}

\newcommand\N{\mathbb{N}}

\newcommand{\QA}[1]{\A_{#1}}

\DeclareMathOperator{\sign}{sign}

\title[On the Jensen convexity of means]{On the Jensen convexity of \\ quasideviation and Bajraktarević means}

\author{Zsolt P\'ales}
\address{Institute of Mathematics, University of Debrecen, Pf.\ 400, 4002 Debrecen, Hungary}
\email{pales@science.unideb.hu}

\author{Pawe\l{} Pasteczka}
\address{Institute of Mathematics, Pedagogical University of Krak\'ow,  Podchor\k{a}\.{z}ych str 2, 30-084 Krak\'ow, Poland}
\email{pawel.pasteczka@up.krakow.pl}

\thanks{The first author was supported by the K-134191 NKFIH Grant.}



\dedicatory{This paper is dedicated to the 70th birthdays of \\ Professors Maciej Sablik and László Székelyhidi.}

\usepackage{color}

\begin{document}
\begin{abstract}
Motivated by the characterization theorem about the Jensen convexity of quasiarithmetic means obtained by the authors in 2021, our main goal is to establish a characterization of the Jensen convexity of quasideviation as well as of Bajraktarević means without \emph{any additional and unnatural regularity assumptions}.
\end{abstract}
\maketitle

\section{Introduction}

Let $I$ denote a nonempty open subinterval of $\R$ throughout this paper. Given $n\in\N$, a function $M_n:I^n\to I$ is called an \emph{$n$-variable mean} if
\Eq{*}{
  \min(x_1,\dots,x_n)\leq M_n(x)\leq
  \max(x_1,\dots,x_n)
}
holds for all $x=(x_1,\dots,x_n)\in I^n$. A function $M:\bigcup_{n=1}^{\infty} I^n\to I$ is said to be a \emph{mean} if, for all $n\in\N$, the restriction $M|_{I^n}$ is an $n$-variable mean. 

The Jensen convexity (Jensen concavity) of means become a key property in the investigation of Hardy-type inequalities (cf.\ \cite{PalPas16,PalPas18a,PalPas18b}).
An $n$-variable mean $M_n:I^n\to I$ is said to be \emph{Jensen convex} if, for all $x,y\in I^n$,
\Eq{*}{
  M_n\Big(\frac{x+y}{2}\Big)\leq\frac{M_n(x)+M_n(y)}{2}.
}
A mean $M:\bigcup_{n=1}^{\infty} I^n\to I$ is said to be \emph{Jensen convex} if, for all $n\in\N$, the $n$-variable mean $M_n:=M|_{I^n}$ is Jensen convex. 

Given $p\in\R$, the \emph{$p^{\mbox{\small\rm th}}$ power mean} or \emph{Hölder mean} $\H_p\colon \bigcup_{n=1}^{\infty} \R_+^n \to \R_+$ is defined by
\Eq{*}{
\H_p(x):= 
\begin{cases} 
\bigg(\dfrac{x_1^p+\cdots+x_n^p}{n} \bigg)^{\frac1p} &\quad \text{ if } p \ne 0, \\[4mm]                                                         
\left(x_1\cdots x_n \right)^{\frac1{n}} &\quad \text{ if } p = 0.
\end{cases}
}
Concerning the convexity of Hölder means we have the following classical result.

\begin{Athm}
Let $n\geq2$ be fixed. Let $p\in\R$, and $I$ be a subinterval of $\R_+$. Then the $n$-variable mean $\H_p|_{I^n}$ is Jensen convex if and only if $p\geq1$.
\end{Athm}

An important generalization of Hölder means is the notion of quasiarithmetic means. Given a continuous strictly monotone function $f \colon I \to \R$, the \emph{quasiarithmetic mean} 
$\A_f:\bigcup_{n=1}^{\infty} I^n \to I$ is defined by
\Eq{QA}{
\A_f(x):= f^{-1} \left( \frac{f(x_1)+\cdots+f(x_n)}{n} \right).
}
If $p\in\R\setminus\{0\}$ and $f(x):=x^p$ for $x\in\R_+$, then $\A_f=\H_p$. If $f(x):=\log x$ for $x\in\R_+$, then $\A_f=\H_0$, therefore, Hölder means are indeed quasiarithmetic means.

The Jensen convexity of quasiarithmetic means have been ultimately characterized by combining the results of the papers \cite{ChuGlaJarJar19} and \cite{PalPas21a}.

\begin{Athm}\label{thm:B}
Let $n\geq2$ be fixed and $f \colon I \to \R$ be continuous and strictly monotone. Then the $n$-variable mean $\A_f|_{I^n}$ is Jensen convex if and only if $f$ is twice continuously differentiable with a nonvanishing first derivative and either $f''$ is identically zero on $I$, or $f''$ is nowhere zero and $f'/f''$ is positive and convex on $I$.
\end{Athm}

Another generalization of Hölder means was introduced by Gini \cite{Gin38}. To recall this definition, for $q,r\in\R$, define the \emph{Gini mean} $\G_{q,r}\colon \bigcup_{n=1}^{\infty} \R_+^n \to \R_+$ by
\Eq{*}{
\G_{q,r} (x):= 
\begin{cases} 
\left(\dfrac{x_1^q+\cdots+x_n^q}{x_1^r+\cdots+x_n^r} \right)^{\frac{1}{q-r}} &\quad \text{ if } q \ne r, \\[4mm]                                                         
\exp\left(\dfrac{x_1^q\log x_1+\cdots+x_n^q\log x_n}{x_1^q+\cdots+x_n^q} \right) &\quad \text{ if } q = r.
\end{cases}
}
The characterization of the convexity of Gini means can easily be deduced from the general results of Losonczi \cite{Los71a} and it reads as follows.

\begin{Athm}
Let $q,r\in\R$. Then the mean $\G_{q,r}$ is Jensen convex if and only if $0\leq\min(q,r)\leq1\leq\max(q,r)$.
\end{Athm}

It is important to emphasize that for a fixed number of the variables, the characterization is different as we have the result of Losonczi and Páles \cite{LosPal97a}:

\begin{Athm}
Let $q,r\in\R$. Then the $2$-variable mean $\G_{q,r}|_{\R_+^2}$ is Jensen convex if and only if $0\leq\min(q,r)\leq1\leq q+r$.  
\end{Athm}

A class of means which includes Gini as well as quasiarithmetic means was discovered by Bajraktarević in the papers \cite{Baj58,Baj63}. Given a positive function $p:I\to\R_+$ and a continuous strictly monotone function $f:I\to\R$, the \emph{Bajraktarević mean} $\B_{f,p}:\bigcup_{n=1}^{\infty} I^n \to I$ is defined by
\Eq{*}{
  \B_{f,p}(x):=f^{-1}\left(\frac{p(x_1)f(x_1)+\cdots+p(x_n)f(x_n)}{p(x_1)+\cdots+p(x_n)}\right).
}
If $q,r\in\R$, $q\neq r$, $f(x):=x^{q-r}$, $p(x):=x^r$ for $x\in\R_+$, or if $q=r\in\R$ and $f(x):=\log(x)$, $p(x):=x^q$ for $x\in\R_+$, then $\B_{f,p}=\G_{q,r}$. Therefore, Gini means form a subclass of Bajraktarević means. On the other hand, if $p$ is a constant function, then one can see that $\B_{f,p}=\A_{f}$ and hence quasiarithmetic means are also included in the class of Bajraktarević means.

The convexity of Bajraktarević means with sufficiently regular generating functions was characterized by the following result of Losonczi \cite{Los71a}.

\begin{Athm}\label{Los71a}
Let $p:I\to\R_+$ be a positive function and $f:I\to\R$ be a differentiable strictly monotone function with a nonvanishing first derivative. Then the Bajraktarević mean $\B_{f,p}$ is convex if and only if the two-variable map
$B_{f,p}:I^2\to\R$ defined by
\Eq{*}{
  B_{f,p}(x,u):=\frac{p(x)(f(x)-f(u))}{p(u)f'(u)}
}
is convex.
\end{Athm}

The notions of deviations and quasideviations were introduced by Daróczy in \cite{Dar71b} and by Páles \cite{Pal82a}, respectively. In what follows, we recall Definition 2.1 and Theorem 2.1 from the paper \cite{Pal82a}. A two-variable function $E:I^2\to\R$ will be called a \emph{quasideviation} if $E$ possesses the following three properties:
\begin{enumerate}[(D1)]
 \item For all $x,u\in I$, the equality $\sign E(x,u)=\sign(x-u)$ holds.
 \item For all $x\in I$, the mapping $I\ni u\mapsto E(x,u)$ is continuous.
 \item For all $x,y\in I$ with $x<y$, the mapping
 \Eq{*}{
   (x,y)\ni u\mapsto\frac{E(x,u)}{E(y,u)}
 }
 is strictly decreasing.
\end{enumerate}
We say that $E$ is a \emph{deviation} (cf.\ \cite{Dar71b}) if $E$ possesses properties (D1), (D2) and, instead of (D3), the following condition 
\begin{enumerate}[(D1)]
 \item[(D3')] For all $x\in I$, the mapping $I\ni u\mapsto E(x,u)$ is strictly decreasing.
\end{enumerate}
It is not difficult to show that every deviation is also a quasideviation. In order to introduce quasideviation means, the following statement is instrumental (cf.\ \cite[Theorem 2.1]{Pal82a}).

\begin{Athm} Let $E:I^2\to\R$ be a quasideviation. Then, for all $n\in\N$ and $x_1,\dots,x_n\in I$, there exists a unique element $u\in I$ such that
\Eq{E(u)}{
  E(x_1,u)+\dots+E(x_n,u)=0.
}
Furthermore, $\min(x_1,\dots,x_n)<u<\max(x_1,\dots,x_n)$ unless $x_1=\dots=x_n$.
\end{Athm}

For $n\in\N$ and $x_1,\dots,x_n\in I$, the solution $u$ of equation \eq{E(u)} is called the \emph{$E$-quasideviation mean of $x_1,\dots,x_n$} and will be denoted by $\D_E(x_1,\dots,x_n)$.

If $f:I\to\R$ is strictly increasing and $p:I\to\R$ is continuous then $E(x,u):=p(x)(f(x)-f(u))$ is a deviation (and also a quasideviation) and a simple computation yields that $\D_E=\B_{f,p}$.

We say that a quasideviation $E\colon I\times I\to\R$ is \emph{normalizable} (cf.\ \cite{Dar71b}) if, for all $x\in I$, the function $u\mapsto E(x,u)$ is differentiable at $x$ and the mapping $x\mapsto\partial_2E(x,x)$ is strictly negative and continuous on $I$. The \emph{normalization} $E^*\colon I\times I\to\R$ of $E$ is defined by 
\Eq{*}{
  E^*(x,u):=-\frac{E(x,u)}{\partial_2E(u,u)}.
}
If $E$ is normalizable, then $E^*$ is also a quasideviation, $\D_E=\D_{E^*}$ and $\partial_2E^*(x,x)=-1$, therefore $(E^*)^*=E^*$.

In the class of deviation means generated by normalizable quasideviations the characterization of the Jensen convexity  follows from a general result of Daróczy \cite{Dar71b}.

\begin{Athm}
Let $E:I\times I\to\R$ be a normalizable quasideviation. Then $\D_E$ is convex if and only if $E^*$ is convex on $I\times I$.
\end{Athm}

Without assuming normalizability, Páles \cite[Theorem~11]{Pal88a} obtained a general theorem, which implies the following result.

\begin{Athm}\label{PZT}
Let $E:I\times I\to\R$ be a quasideviation. Then $\D_E$ is convex if and only if there exist two functions $a,b:I\times I\to\R$ such that, for all $x,y,u,v\in I$,
\Eq{ab}{
  E\Big(\frac{x+y}{2},\frac{u+v}{2}\Big)
  \leq a(u,v)E(x,u)+b(u,v)E(y,v).
}
\end{Athm}

Motivated by the characterization theorem about the Jensen convexity of quasiarithmetic means, our aim is to establish a characterization of the Jensen convexity of quasideviation and Bajraktarević means without \emph{any additional regularity assumptions}, that is to generalize Theorems E and G. The main starting point of our approach will be Theorem H.

\section{Main results}

The following auxiliary result will be needed in the sequel.

\begin{lem}\label{lem:max}
Let $I \subset \R$ be an open interval and $f \colon I \to\R$. If $f(\frac{x+y}{2})\le f(x)$ for all $x,y \in I$, then $f$ is constant. 
\end{lem}
\begin{proof}
 Take any element $x$ of $I$ and $\varepsilon>0$ such that $x-2\varepsilon,x+2\varepsilon \in I$. Now, for all $\delta \in (-\varepsilon,\varepsilon)$, we have that $x-\delta,x+\delta,x+2\delta\in (x-2\varepsilon,x+2\varepsilon)\subseteq I$, and 
 \Eq{*}{
f(x)= f\Big(\tfrac{(x+\delta)+(x-\delta)}2\Big)\le f(x+\delta)=f\Big(\tfrac{x+(x+2\delta)}2\Big)\le f(x).
}
Consequently, $f(x)=f(x+\delta)$ for all $\delta \in (-\varepsilon,\varepsilon)$. Thus $f$ is differentiable at $x$ and $f'(x)=0$. 

Since $x$ was arbitrary, we obtain that $f$ is differentiable on $I$ and $f'$ is identically zero on it. This implies that $f$ is constant.
\end{proof}

To simplify the formulation of some of the results and also the proofs, we introduce the following regularity property: A function $f\colon I\to\R$ is called \emph{nearly differentiable} if, at every point of $I$, it has left and right derivatives and the set of those points where these one-sided derivatives are different is at most countable. It is well-known that convex functions admit this regularity property.

\begin{thm}\label{thm:QDconv} Let $E:I^2\to\R$ be a quasideviation. Then the following conditions are equivalent to each other:
\begin{enumerate}[{\rm(i)}]
 \item The quasideviation mean $\D_E$ is Jensen convex.
 \item For all $u\in I$, the map $x\mapsto E(x,u)$ has a positive right-derivative at $x=u$, denoted by $\partial_1^+E(u,u)$, and the mapping $E^+\colon I^2\to\R$ defined by $E^+(x,u):=\frac{E(x,u)}{\partial_1^+E(u,u)}$ is convex on $I^2$.
 \item For all $u\in I$, the map $x\mapsto E(x,u)$ has a positive left-derivative at $x=u$, denoted by $\partial_1^-E(u,u)$, and the mapping $E^-\colon I^2\to\R$ defined by $E^-(x,u):=\frac{E(x,u)}{\partial_1^-E(u,u)}$ is convex on $I^2$.
\end{enumerate}
Moreover, if any of the above equivalent conditions is satisfied then the quasideviation $E$ and quasideviation mean $\D_E$ possess the following properties:
\begin{enumerate}[(a)]
 \item $E$ is continuous on $I^2$.
 \item For all $u\in I$, the function $I\ni x\mapsto E(x,u)$ is convex.
 \item The function $I\ni u\mapsto (\partial_1^-E(u,u),\partial_1^+E(u,u))$ is continuous.
 \item The map $I\ni u \mapsto \frac{\partial_1^+E(u,u)}{\partial_1^-E(u,u)}$ is constant.
 \item $\D_E$ is nonsmaller than the arithmetic mean. 
\end{enumerate}
\end{thm}

\begin{proof}
Now assume that the condition (i) is satisfied.
In view of Theorem~\ref{PZT}, we know that the quasideviation mean $\D_E$ is Jensen convex if and only if there exists two functions $a,b \colon I^2 \to \R$ such that \eq{ab} holds.
Interchanging the pair $(x,u)$ with $(y,v)$, it follows that
\Eq{*}{
E\Big(\frac{x+y}2,\frac{u+v}2\Big)\le a(v,u) E(y,v)+b(v,u)E(x,u),
\qquad x,y,u,v \in I.
}
Adding up the above inequality to \eq{ab} side by side, we get
\Eq{E1}{
E\Big(\frac{x+y}2,\frac{u+v}2\Big)\le c(u,v) E(x,u)+c(v,u)E(y,v),
\qquad x,y,u,v \in I,
}
where $c\colon I^2\to\R$ is defined by
\Eq{*}{
  c(u,v):=\frac{a(u,v)+b(v,u)}{2},\qquad u,v \in I.
}
Putting $y:=x$ and $v:=u$ into \eq{E1} we obtain
\Eq{*}{
E(x,u)\le 2c(u,u) E(x,u),\qquad x,u \in I.
}
Then, by property (D1) of quasideviations, we get that, for each $u \in I$, the factor $E(\cdot,u)$ takes both positive and negative values, thus we can conclude that
\Eq{*}{
c(u,u)=\frac12,\qquad u \in I.
}
In the next step, substituting $u:=v$ into \eq{E1}, we get 
\Eq{*}{
E\Big(\frac{x+y}2,u\Big)\le \frac{E(x,u)+E(y,u)}2, \qquad x,y,u \in I.
}
Therefore, for each fixed $u\in I$, the function $I \ni x \mapsto E(x,u)$ is Jensen convex. On the other hand, this function is bounded from above by $0$ on the open interval $(-\infty,u)\cap I$. Thus, in view of the Bernstein-Doetsch Theorem \cite{BerDoe15}, it is convex. As a consequence, all such maps are nearly differentiable. Thus, for all $(x,u) \in I^2$, the one-sided partial derivatives $\partial_1^+E(x,u)$ and $\partial_1^-E(x,u)$ exist. By applying property (D1) of quasideviations, we also get 
\Eq{E:partial+}{
 \partial_1^+E(u,u)\geq\partial_1^-E(u,x)>0,\qquad u\in I.
}

Next, putting $y:=v$ into \eq{E1} we get 
\Eq{*}{
E\Big(\frac{x+v}2,\frac{u+v}2\Big)\le c(u,v) E(x,u),\qquad x,u,v \in I.
}
Now, using (D1), we can obtain the double inequality
\Eq{E3}{
\frac{E\big(\frac{x_2+v}2,\frac{u+v}2\big)}{E(x_2,u)}\le c(u,v) \le \frac{E\big(\frac{x_1+v}2,\frac{u+v}2\big)}{E(x_1,u)}, 
\quad x_1,x_2,u,v \in I\text{ with }x_1<u<x_2.
}
Since $E$ vanishes on the diagonal (by (D1)), we obtain
\Eq{*}{
\lim_{x_1\to u^-}&\frac{E\big(\frac{x_1+v}2,\frac{u+v}2\big)}{E(x_1,u)}\\
&=\frac12\lim_{x_1\to u^-}\frac{E\big(\frac{x_1+v}2,\frac{u+v}2\big)-E\big(\frac{u+v}2,\frac{u+v}2\big)}{\frac{x_1+v}2-\frac{u+v}2}\frac{x_1-u}{E(x_1,u)-E(u,u)}\\
&=\frac{\partial_1^-E(\frac{u+v}2,\frac{u+v}2)}{2\partial_1^-E(u,u)}.
}
Similarly
\Eq{*}{
\lim_{x_2\to u^+}\frac{E\big(\frac{x_2+v}2,\frac{u+v}2\big)}{E(x_1,u)}=\frac{\partial_1^+E(\frac{u+v}2,\frac{u+v}2)}{2\partial_1^+E(u,u)}.
}
Upon taking the limits $x_1 \to u^-$ and $x_2 \to u^+$ in the inequalities \eq{E3}, in view of the just proved equalities, we arrive at
\Eq{E4}{
\frac{\partial_1^+E(\frac{u+v}2,\frac{u+v}2)}{2\partial_1^+E(u,u)}
\le c(u,v) \le \frac{\partial_1^-E(\frac{u+v}2,\frac{u+v}2)}{2\partial_1^-E(u,u)},\qquad u,v \in I.
}
By \eq{E:partial+}, we can rewrite this inequality in the following way
\Eq{*}{
\frac{\partial_1^+E(\frac{u+v}2,\frac{u+v}2)}{\partial_1^-E(\frac{u+v}2,\frac{u+v}2)}
\le \frac{\partial_1^+E(u,u)}{\partial_1^-E(u,u)},\qquad u,v \in I.
}
Applying Lemma~\ref{lem:max} to the function $f(u):=\frac{\partial_1^+E(u,u)}{\partial_1^-E(u,u)}$, we can see that $f$ is constant, thus there exists a constant $\alpha \in \R$  such that $\partial_1^+E(u,u)=\alpha \partial_1^-E(u,u)$ for all $u \in I$.
Obviously $\alpha\geq 1$ since $\partial_1^+E(u,u)\geq\partial_1^-E(u,u)>0$. Therefore,
\Eq{*}{
\frac{\partial_1^+E(\frac{u+v}2,\frac{u+v}2)}{2\partial_1^+E(u,u)}=\frac{\alpha\partial_1^-E(\frac{u+v}2,\frac{u+v}2)}{2\alpha\partial_1^-E(u,u)}=\frac{\partial_1^-E(\frac{u+v}2,\frac{u+v}2)}{2\partial_1^-E(u,u)},\qquad u,v \in I.
}
Consequently, the inequalities in \eq{E4} yield
\Eq{*}{
c(u,v)=\frac{\partial_1^+E(\frac{u+v}2,\frac{u+v}2)}{2\partial_1^+E(u,u)} = \frac{\partial_1^-E(\frac{u+v}2,\frac{u+v}2)}{2\partial_1^-E(u,u)},\qquad u,v \in I.
}
Thus, using \eq{E1}, the inequality \eq{E:partial+} implies the following Jensen convexity-type properties:
\Eq{E5}{
\frac{E(\frac{x+y}2,\frac{u+v}2)}{\partial_1^+E(\frac{u+v}2,\frac{u+v}2)}\le \frac12 \left( \frac{E(x,u)}{\partial_1^+E(u,u)}+\frac{E(y,v)}{\partial_1^+E(u,u)} \right),\qquad x,y,u,v \in I
}
and
\Eq{*}{
\frac{E(\frac{x+y}2,\frac{u+v}2)}{\partial_1^-E(\frac{u+v}2,\frac{u+v}2)}\le \frac12 \left( \frac{E(x,u)}{\partial_1^-E(u,u)}+\frac{E(y,v)}{\partial_1^-E(u,u)} \right),\qquad x,y,u,v \in I.
}
Equivalently, both $E^+$ and $E^-$ are Jensen convex over $I^2$. On the other hand, the function $E$ and hence also $E^+$ and $E^-$ are bounded from above by zero over the open set $\{(x,u)\in I^2\mid x<u\}$. Therefore, in view of Bernstein--Doetsch theorem, the Jensen convexity implies the convexity of both $E^+$ and $E^-$, i.e., the conditions (ii) and (iii) hold, respectively.

To show the converse implications assume that (ii) holds, that is, for all $u\in I$, the map $x\mapsto E(x,u)$ has a positive right-derivative at $x=u$ and $E^+$ is convex over $I^2$. Then \eq{E1} is satisfied with 
\Eq{*}{
c(u,v):=\frac{\partial_1^+E(\frac{u+v}2,\frac{u+v}2)}{2\partial_1^+E(u,u)},
}
which, by applying Theorem~\ref{PZT}, implies that $\D_E$ is Jensen convex. The proof of the implication (iii)$\Longrightarrow$(i) is analogous.

To prove the last statements of the theorem, assume that (i) (and hence (ii), (iii)) holds. As we have seen it in the proof, this implies that $E$ is convex in its first variable, i.e., (b) is valid. We have also verified assertion (d). It follows from (ii) that the function $E^+$ is convex, and hence it is continuous on $I^2$. 

To prove assertion (c), let $u_0\in I$ be fixed and chose $x\in I\setminus\{u_0\}$. Then, using also property (D2) of quasideviations, its follows that the map
\Eq{*}{
  u\mapsto\frac{E(x,u)}{E^+(x,u)}=\partial_1^+E(u,u)
}
is continuous at $u_0$. This proves that the map $I\ni u\mapsto \partial_1^+E(u,u)$ is continuous. Similarly, we can see that the map $I\ni u\mapsto \partial_1^-E(u,u)$ is also continuous and hence assertion (c) is valid. 

Using the equality $E(x,u)=\partial_1^+E(u,u)\cdot E^+(x,u)$, the continuity of $E^+$ (which is a consequence of its convexity) and property (c), we can conclude that assertion (a) is also valid. 

It easily follows from property (b) of the quasideviation $E$, that (e) holds. Indeed, if for all $u\in I$, the function $E(\cdot,u)$ is convex, then there there exists a function $h:I\to\R$ such that
\Eq{*}{
  h(u)(x-u)\leq E(x,u) \qquad(x,u\in I).
}
According to the results of the paper \cite[Theorem 7, condition (iv)]{Pal88a}, it follows that $\D_E$ is nonsmaller than the arithmetic mean. 
\end{proof}

\begin{thm}\label{thm:Eab} Let $E\colon I^2\to\R$ be a quasideviation and $\alpha,\beta \in (0,\infty)$. Define $E_{\alpha,\beta}\colon I^2\to\R$ by
\Eq{Eabdef}{
 E_{\alpha,\beta}(x,u):
 =\begin{cases} 
          \alpha E(x,u) &\text{ for }x\le u, \\[1mm]
          \beta E(x,u) &\text{ for }x>u.
         \end{cases}
}
Then $E_{\alpha,\beta}$ is a quasideviation. If, additionally, $\D_E$ is Jensen convex and $\alpha\leq\beta$, then so is $\D_{E_{\alpha,\beta}}$. \\\indent Furthermore, if $E$ is differentiable in the sense of Gateaux at every point of the diagonal of $I^2$ and the map $u\mapsto \partial_1E(u,u)$ is continuous, then $\D_{E_{\alpha,\beta}}$ is Jensen convex if and only if $\D_E$ is Jensen convex and $\alpha\leq\beta$.
\end{thm}

\begin{proof} The properties (D1) and (D2) of quasideviations are obviously satisfied. To check (D3), let $x,y\in I$ with $x<y$. Then, for all $u\in(x,y)$, we have
\Eq{*}{
  \frac{E_{\alpha,\beta}(x,u)}{E_{\alpha,\beta}(y,u)}
  =\frac{\alpha E(x,u)}{\beta E(y,u)}.
}
The right hand side is strictly decreasing function of $u$ because $E$ is a quasideviation, therefore, so is the left hand side, which shows that $E_{\alpha,\beta}$ also possess property (D3).

Assume now that $\D_E$ is Jensen convex and $\alpha\leq\beta$. Then,
\Eq{*}{
  E_{\alpha,\beta}(x,u)=\max(\alpha E(x,u),\beta E(x,u)), \qquad(x,u)\in I^2,
}
and, according to condition (ii) of Theorem~\ref{thm:QDconv}, the function $E^+$ is convex over $I^2$. On the other hand, for $(x,u)\in I^2$, 
\Eq{*}{
 E^+_{\alpha,\beta}(x,u)
 &=\frac{E_{\alpha,\beta}(x,u)}{\partial_1^+E_{\alpha,\beta}(u,u)}
 =\frac{\max(\alpha E(x,u),\beta E(x,u))}{\beta\partial_1^+E(u,u)}\\
 &=\frac1\beta\max\big(\alpha E^+(x,u),\beta E^+(x,u)\big).
}
Therefore,
\Eq{E+ab}{
  E^+_{\alpha,\beta}=\frac1\beta\max\big(\alpha E^+,\beta E^+\big),
}
which shows that $E^+_{\alpha,\beta}$ is the maximum of two convex functions, and hence, itself is convex. Thus, condition (ii) of Theorem~\ref{thm:QDconv} holds for $E_{\alpha,\beta}$ and hence $\D_{E_{\alpha,\beta}}$ is Jensen convex.

Now assume that $E$ is differentiable in the sense of Gateaux at every point of the diagonal of $I^2$, the map $I\ni u\mapsto \partial_1E(u,u)$ is continuous and $\D_{E_{\alpha,\beta}}$ is Jensen convex. Then $E^+_{\alpha,\beta}$ is convex. 

To prove that $\alpha\leq\beta$, let $u\in I$ be fixed.
Then we have that
\Eq{*}{
  \alpha\partial_1 E(u,u)=\partial_1^-E_{\alpha,\beta}(u,u)
  \leq \partial_1^+E_{\alpha,\beta}(u,u)=\beta\partial_1 E(u,u).
} 
Since $\partial_1 E(u,u)>0$, it follows that $\alpha\leq\beta$.

The Jensen convexity of $\D_{E_{\alpha,\beta}}$ implies that $E^+_{\alpha,\beta}$ is convex. In view of formula \eq{E+ab}, we can see that $E^+$ is convex on both triangles $\Delta^+:=\{(x,u)\in I^2\mid x\leq u\}$ and $\Delta^-:=\{(x,u)\in I^2\mid x\geq u\}$. To prove that $E^+$ is convex on $I^2=\Delta^+\cup\Delta^-$, it suffices to show that $E^+$ is convex along any line which crosses the diagonal of $I^2$.

Let $u\in I$ be fixed and let $(0,0)\neq(v,w)\in\R^2$ be arbitrary. Then the line $\R\ni t\mapsto (u+tv,u+tw)$ crosses the diagonal of $I^2$ at $(u,u)$. We are going to show that the function $e:T\to\R$ defined by $e(t):=E^+(u+tv,u+tw)$ is convex over the interval $T:=\{t\in\R\mid (u+tv,u+tw)\in I^2\}$. The convexity of $E^+$ over the triangles $\Delta^+$ and $\Delta^-$ implies that $e$ is convex over the subintervals $T_-:=(-\infty,0]\cap T$ and $T_+:=[0,\infty)\cap T$. On the other hand, using the continuity of the map $u\mapsto \partial_1E(u,u)$, we can get that
\Eq{*}{
  \lim_{t\to 0} \frac{e(t)-e(0)}{t}
  &=\lim_{t\to 0} \frac{E^+(u+tv,u+tw)}{t}
  =\lim_{t\to 0} \frac{E(u+tv,u+tw)}{\partial_1 E(u+tw,u+tw)t}\\
  &=\frac{1}{\partial_1 E(u,u)}
   \lim_{t\to 0} \frac{E(u+tv,u+tw)-E(u,u)}{t}.
}
By the Gateaux differentiability assumption on $E$, the limit on the right hand side exists, therefore, $e$ is differentiable at $t=0$.
This property of $e$ together with its convexity over the subintervals $T_-$ and $T_+$ imply that $e$ is convex over $T$. Therefore, we have proved that $E^+$ is convex on $I^2$ and hence, the mean $\D_{E}$ is Jensen convex.
\end{proof}

\begin{cor}
Let $f \colon I \to \R$ be a continuous, strictly increasing function and $\alpha,\beta \in (0,\infty)$ with $\alpha\le \beta$. Then the function $E_{\alpha,\beta} \colon I^2 \to \R$ given by
 \Eq{Edef}{
 E_{\alpha,\beta}(x,u):=\begin{cases} 
          \alpha (f(x)-f(u)) &\text{ for }x\le u; \\
          \beta (f(x)-f(u)) &\text{ for }x>u
         \end{cases}
}
is a quasideviation. 
Furthermore, $\D_{E_{\alpha,\beta}}$ is Jensen convex if and only if $\alpha \le \beta$, $f$ is twice differentiable with a positive derivative and 
\Eq{E:varphi} 
{\text{either $f''$ is nonvanishing and $\frac{f'}{f''}$ is positive and convex or $f''\equiv 0$.}
}
\end{cor}

\begin{proof} Define $E:I^2\to\R$ by $E(x,u):=f(x)-f(u)$. Then $E$ is a deviation and hence it is a quasideviation. Thus, by the first statement of Theorem~\ref{thm:Eab}, we can see that $E_{\alpha,\beta}$ is a quasideviation.

Assume first that $\D_{E_{\alpha,\beta}}$ is Jensen convex. Then, by assertion (b) of Theorem~\ref{thm:QDconv}, for all $u\in I$, the map $x\mapsto E_{\alpha,\beta}(x,u)$ is convex on $I$. This implies that $\alpha f-\alpha f(u)$ is convex on $(-\infty,u)\cap I$ for all $u\in I$, and hence, $f$ is convex on $I$. Therefore, $f$ is nearly differentiable. We can now get, for all $u\in I$, that $\partial_1^-E(u,u)=\alpha f'_-(u)$ and $\partial_1^+E(u,u)=\beta f'_+(u)$. In view of assertion (b) of Theorem~\ref{thm:QDconv}, the ratio function 
\Eq{*}{
  u\mapsto \frac{\partial_1^+E(u,u)}{\partial_1^+E(u,u)}
  =\frac{\beta f'_+(u)}{\alpha f'_-(u)}
}
is constant on $I$. Since, except countably many values of $u$, we have that $f'_+(u)=f'_-(u)$, therefore the value of the above ratio equals the constant $\beta/\alpha$. Thus, for all $u\in I$, we obtain that $f'_+(u)=f'_-(u)$, which proves the differentiability of $f$ at every element of $I$. Thus $E$ is also differentiable over $I^2$, it is Gateaux differentiable at the diagonal points of $I^2$. Thus, in view of Theorem~\ref{thm:Eab}, it follows that $\alpha\leq\beta$ and that the mean $\D_E$ is Jensen convex. According to Theorem~\ref{thm:B}, it follows that $\D_E$ is Jensen convex if and only if $f$ is twice differentiable with a positive derivative and \eq{E:varphi} holds.

Now assume to the converse that $E_{\alpha,\beta}$ is of the form \eq{Edef} for some $\alpha,\beta \in(0,+\infty)$ with $\alpha\le \beta$ and a function $f$ which satisfies \eq{E:varphi}. Then, by Theorem~\ref{thm:B}, $\QA{f}=\D_{E}$ is convex and, due to Theorem~\ref{thm:Eab}, so is the mean $\D_{E_{\alpha,\beta}}$. 
\end{proof}

\section{The case of Bajraktarević means}

In what follows, the spaces of $k$ times continuously differentiable functions and $k$ times continuously differentiable functions with a nonvanishing first derivative (which are defined on the open interval $I$) will be denoted by $\calC^k(I)$ and $\calC^{k\#}(I)$, respectively.

\begin{thm}\label{thm:Bconv} Let $f:I\to\R$ be a strictly monotone and continuous function and $p\colon I\to\R_+$ be a positive function. Then the following conditions are equivalent to each other:
\begin{enumerate}[{\rm(i)}]
 \item The Bajraktarević mean $\B_{f,p}$ is Jensen convex.
 \item $f \in \calC^{1\#}(I)$ and the mapping $B_{f,p}\colon I^2\to\R$ defined by 
 \Eq{*}{B_{f,p}(x,u):=\frac{p(x)(f(x)-f(u))}{p(u)f'(u)}} is convex on $I^2$.
 \item $f \in \calC^{2\#}(I)$, $p \in \calC^1(I)$, and for all $x,y,u,v\in I$,
 \Eq{*}{
    \frac{p(y)(f(y)-f(v))}{p(v)f'(v)}
    &\geq \frac{p(x)(f(x)-f(u))}{p(u)f'(u)}
    +\frac{(pf)'(x)-f(u)p'(x)}{(pf')(u)}(y-x)\\
    &\quad+p(x)\frac{(f(u)-f(x))\cdot (pf')'(u)-(pf')(u)\cdot f'(u)}{(pf')(u)^2}(v-u).
 }
 \item $f \in \calC^{2\#}(I)$, $p \in \calC^1(I)$, and for all $x,y,u,v\in I$,
 \Eq{*}{
   &\bigg(\frac{(pf)'(x)-f(u)p'(x)}{(pf')(u)}-\frac{(pf)'(y)-f(v)p'(y)}{(pf')(v)}\bigg)(x-y)\\
   &\qquad+\bigg(p(x)\frac{(f(u)-f(x)) (pf')'(u)-(pf')(u)f'(u)}{(pf')(u)^2}\\&\qquad\qquad-p(y)\frac{(f(v)-f(y))\cdot (pf')'(v)-(pf')(v)\cdot f'(v)}{(pf')(v)^2}\bigg)(u-v)\geq0.
 }
\end{enumerate}
\end{thm}

\begin{proof} Without loss of generality, we may assume that $f$ is increasing. Define the quasideviation $E\colon I^2\to\R$ by $E(x,u):=p(x)(f(x)-f(u))$. Then we have that $\B_{f,p}=\D_E$.

Assume that $\B_{f,p}=\D_E$ is Jensen convex. Then, according to assertion (a) of Theorem~\ref{thm:QDconv}, we get that $E$ is convex in its first variable. That is, for all $u\in I$, the function $pf-f(u)p$ is convex and hence it is nearly differentiable. Let $u,v$ be distinct elements of $I$, then 
\Eq{*}{
p=\frac{(pf-f(u)p)-(pf-f(v)p)}{f(v)-f(u)},
}
which shows that $p$ is also nearly differentiable. We also have that
\Eq{*}{
  f=\frac{pf-f(u)p}{p}+f(u),
}
which shows that $f$ is also nearly differentiable.

In view of these properties, for all $u\in I$, we can obtain
\Eq{*}{
  \partial_1^+E(u,u)&=(pf-f(u)p)_+'(u)\\
  &=p_+'(u)f(u)+p(u)f_+'(u)-f(u)p_+'(u)=p(u)f_+'(u).
}
Similarly,
\Eq{*}{
  \partial_1^-E(u,u)=p(u)f_-'(u).
}
By assertion (a) of Theorem~\ref{thm:QDconv}, the ratio function $u\mapsto\frac{\partial_1^+E(u,u)}{\partial_1^-E(u,u)}$ is constant, therefore,
$f_+'=cf_-'$ for some constant $c\in\R$. On the other hand, $f$ is differentiable nearly everywhere, hence, $c=1$, which yields that $f$ is differentiable everywhere with a positive derivative. Assertion (ii) of Theorem~\ref{thm:QDconv} now gives us that the function $B_{f,p}$ defined in assertion (ii) is convex. Thus, we have proved the equivalence of assertions (i) and (ii).

Assume now that (ii) holds. It follows from the convexity of $B_{f,p}$ that, for all $x\in I$, the map $u\mapsto B_{f,p}(x,u)$ is convex. Therefore, it is nearly differentiable. For $x,u\in I$, we have that
\Eq{*}{
  f'(u)=\frac{p(x)(f(x)-f(u))}{p(u)B_{f,p}(x,u)}. 
}
For any fixed $x\in I$, the function on the right hand side is nearly differentiable with respect to $u$. Consequently, $f'$ is also nearly differentiable, in particular, $f'$ is continuous.

Using the convexity $B_{f,p}$ again, we can obtain that there exist two functions $r,s\colon I^2\to\R$ such that
\Eq{IBfp}{
  B_{f,p}(y,v)-B_{f,p}(x,u)\geq r(x,u)(y-x)+s(x,u)(v-u), \quad x,y,u,v\in I. 
}
After substituting $y:=x$, inequality \eq{IBfp} implies that
\Eq{*}{
  \frac{(f(x)-f(v))((pf')(u)-(pf')(v))-(pf')(v)(f(v)-f(u))}{(pf')(v)(pf')(u)}\geq \frac{s(x,u)}{p(x)}(v-u).
}
If $v>u$, then dividing the inequality by $(v-u)$ side by side, then taking the right limit as $v\downarrow u$, we get
\Eq{*}{
 \frac{(f(u)-f(x))\cdot (pf')_+'(u)-(pf')(u)\cdot f'(u)}{(pf')(u)^2}
 \geq \frac{s(x,u)}{p(x)}, \qquad x,u\in I.
}
Repeating the above argument for $v<u$, we get that
\Eq{*}{
 \frac{s(x,u)}{p(x)}
 \geq\frac{(f(u)-f(x))\cdot (pf')_-'(u)-(pf')(u)\cdot f'(u)}{(pf')(u)^2}, \qquad x,u\in I.
}
Binding the above two inequalities, it follows that
\Eq{*}{
  (f(u)-f(x))\cdot [(pf')_+'(u)-(pf')_-'(u)]\geq0, \qquad x,u\in I.
}
Since $x$ is arbitrary, this inequality can hold only if 
$(pf')_+'(u)=(pf')_-'(u)$ for all $u\in I$, which proves that $pf'$ is differentiable everywhere. It follows from this property that, for all $x,u\in I$,
\Eq{D2B}{
  s(x,u)=p(x)\frac{(f(u)-f(x))\cdot (pf')'(u)-(pf')(u)\cdot f'(u)}{(pf')(u)^2}=\partial_2B_{f,p}(x,u).
}

Now taking \eq{IBfp} for $v:=u$, we get
\Eq{*}{
  \frac{p(y)(f(y)-f(u))-p(x)(f(x)-f(u))}{(pf')(u)}\geq r(u,x)(y-x), \qquad x,y,u\in I.
}
Therefore, for all $x,y,u\in I$ with $y>x$ we obtain
\Eq{*}{
r(u,x)&\leq\frac{p(y)(f(y)-f(u))-p(x)(f(x)-f(u))}{(pf')(u) (y-x)}\\
&=\frac{1}{(pf')(u)}\bigg(\frac{(pf)(y)-(pf)(x)}{y-x}-f(u)\frac{p(y)-p(x)}{y-x}\bigg).
}
Since both $p$ and $f$ are nearly differentiable, we can take the limit $y \searrow x$ to obtain
\Eq{soL}{
\frac{(pf)'_+(x)-f(u)p'_+(x)}{(pf')(u)}\ge r(u,x), \qquad x,u \in I.
}
Similarly, for all $x,y,u \in I$ with $y<x$, one gets
\Eq{*}{
r(u,x)&\ge\frac{1}{(pf')(u)}\bigg(\frac{(pf)(y)-(pf)(x)}{y-x}-f(u)\frac{p(y)-p(x)}{y-x}\bigg),
}
which, in the limiting case as $y \nearrow x$ leads us to the inequality
\Eq{soR}{
r(u,x)\ge \frac{(pf)'_-(x)-f(u)p'_-(x)}{(pf')(u)}, \qquad x,u \in I.
}
From the inequalities \eq{soL} and \eq{soR}, we can conclude that
\Eq{*}{
\frac{(pf)'_+(x)-f(u)p'_+(x)}{(pf')(u)}&\ge \frac{(pf)'_-(x)-f(u)p'_-(x)}{(pf')(u)}, \qquad x,u \in I.
}
We know that $(pf')(u) >0$, thus we can obtain that
\Eq{*}{
(pf)'_+(x)-f(u)p'_+(x)&\ge (pf)'_-(x)-f(u)p'_-(x), \qquad x,u \in I.
}
By using the differentiability of $f$, for all $x,u \in I$, it follows that 
\Eq{*}{
p'_+(x)f(x)+p(x)f'(x)-f(u)p'_+(x)\ge p'_-(x)f(x)+p(x)f'(x)-f(u)p'_-(x),
}
which can equivalently be rewritten as
\Eq{*}{
(f(x)-f(u))(p'_+(x)-p'_-(x)) \ge 0, \qquad x,u \in I.
}
Since $u$ is arbitrary and $f$ is strictly monotone, this inequality can only hold if $p'_+(x)-p'_-(x)=0$ for all $x \in I$, which yields differentiability of $p$ on the interval $I$. However, we have already proved that $pf'$ is differentiable, therefore $f$ must be twice differentiable.

Therefore, the upper and lower bounds for the function $r$ given by \eq{soL} and \eq{soR} are equal to each other, whence we get that
\Eq{D1B}{
 r(u,x)=\frac{(pf)'(x)-f(u)p'(x)}{(pf')(u)}=\partial_1B_{f,p}(x,u), \qquad x,u \in I.
}

The differentiability of $p$ and the twice differentiability of $f$ imply that $B_{f,p}$ is differentiable. On the other hand, 
it is well known that the partial derivatives of a differentiable convex function are continuous. Therefore, $\partial_1B_{f,p}$ and $\partial_2B_{f,p}$ are continuous over $I^2$.

In view of formula \eq{D1B}, for all $x,u,\in I$ with $x\neq u$, we can obtain that
\Eq{*}{
p'(x)=\frac{\partial_1B_{f,p}(x,u)(pf')(u)-(pf')(x)}{f(x)-f(u)}.
}
This shows that $p'$ is continuous everywhere except at $x=u$. But, since $u$ was an arbitrary element of $I$, we get that $p'$ is continuous on $I$ and hence it belongs to $\calC^{1}(I)$.

Using formula \eq{D2B}, for all $x,u,\in I$ with $x\neq u$, we can get that
\Eq{*}{
f''(u)=\frac{1}{p(u)}\bigg(\frac{(pf')(u)^2\partial_2 B_{f,p}(x,u)+p(x)(pf')(u)f'(u)}{p(x) (f(u)-f(x))}-(p'f')(u)\bigg),
}
which shows that $f''$ is continuous everywhere except at $u=x$. Since $x$ was arbitrary, this implies that $f''$ is continuous on $I$ and hence it belongs to $\calC^{2\#}(I)$.

Now the inequality \eq{IBfp} can be seen to be equivalent to condition (iii), hence the implication (ii)$\Rightarrow$(iii) is verified. On the other hand, if (iii) holds, then $B_{f,p}$ is the pointwise supremum of affine functions and hence it is convex, i.e., (ii) holds as well. The last condition expresses the monotonicity of the gradient of $B_{f,p}$, i.e., that, for all $x,y,u,v\in I$, the inequality
\Eq{*}{
    (\partial_1 B_{f,p}(x,u)\!-\!\partial_1 B_{f,p}(y,v))(x\!-\!y)
    +(\partial_2 B_{f,p}(x,u)\!-\!\partial_2 B_{f,p}(y,v))(u\!-\!v)\geq0
}
holds, which is also known to be equivalent to the convexity of $B_{f,p}$.
\end{proof}

\section{Convexity of Gini Means in Subintervals}

For $q,r\in\R$, we need to introduce the following notations.
\Eq{*}{
  \gamma_{q,r}(t):=
  \begin{cases}
  \dfrac{t^q-t^r}{q-r} &\mbox{if } q\neq r,\\[3mm]
  t^q\log t &\mbox{if } q=r,
  \end{cases}\qquad t\in\R_+,
}
and
\Eq{*}{
  \beta_{q,r}:=
  \begin{cases}
   \bigg(\dfrac{q(q-1)}{r(r-1)}\bigg)^{\frac{1}{q-r}}
   &\mbox{if } q\neq r \mbox{ and } qr(q-1)(r-1)>0,\\[3mm]
   \exp\bigg(\dfrac{1}{q}+\dfrac{1}{q-1}\bigg)  &\mbox{if } q=r \mbox{ and } q(q-1)\neq0.
  \end{cases}
}

\begin{thm}
Let $q,r\in\R$, $0<a<b<\infty$. Then the following three assertions are equivalent to each other.
\begin{enumerate}[(i)]
 \item The mean $\G_{q,r}$ is Jensen convex on the interval $(a,b)$. 
 \item The function $\gamma_{q,r}$ is convex on the interval $\big[\frac{a}{b},\frac{b}{a}\big]$.
 \item One of the following conditions is valid:
 \begin{enumerate}[(1)]
  \item $0\leq\min(q,r)\leq1\leq\max(q,r)$;
  \item $\max(q,r)<1\leq q+r$ and $\beta_{q,r}\leq\frac{a}{b}$;
  \item $\min(q,r)\leq 0$, $1\leq q+r$ and $\beta_{q,r}\geq\frac{b}{a}$;
  \item $1\leq \min(q,r)$ and $\beta_{q,r}\geq\frac{b}{a}$.
 \end{enumerate}
\end{enumerate}
\end{thm}
 
\begin{proof} Let $r<q$ in the subsequent argument. The cases $q=r$ and $q<r$ can be dealt with analogously and therefore, they are left to the reader.

Define $f(x):=x^{q-r}$ and $p(x):=x^r$ for $x\in\R_+$.
Then the Bajraktarević mean $\B_{f,p}$ equals the Gini mean $\G_{q,r}$. Therefore, to characterize the Jensen 
convexity of $\G_{q,r}$ on $(a,b)$, we need to describe the Jensen convexity of $\B_{f,p}$ on $(a,b)$. According to Theorem~\ref{Los71a} or to our Theorem~\ref{thm:Bconv} this property is equivalent to the convexity of the following mapping
\Eq{*}{
  (a,b)^2\ni (x,u)\mapsto \frac{p(x)(f(x)-f(u))}{p(u)f'(u)}
  &= \frac{x^r(x^{q-r}-u^{q-r})}{(q-r)u^{r}u^{q-r-1}}
  \\ &=\frac{u}{q-r}\bigg(\Big(\frac{x}{u}\Big)^q-\Big(\frac{x}{u}\Big)^r\bigg)=u\,\gamma_{q,r}\Big(\frac{x}{u}\Big).
}
That is, $\G_{q,r}$ is Jensen convexity on $(a,b)$ if and only if, for all $x,u,y,v\in(a,b)$ and $t\in[0,1]$,
\Eq{*}{
  (tu+(1-t)v)\,\gamma_{q,r}\Big(\frac{tx+(1-t)y}{tu+(1-t)v}\Big)
  \leq tu\,\gamma_{q,r}\Big(\frac{x}{u}\Big)
  +(1-t)v\,\gamma_{q,r}\Big(\frac{y}{v}\Big).
}
This inequality is equivalent to
\Eq{*}{
  \gamma_{q,r}\Big(\frac{tu}{tu+(1-t)v}\frac{x}{u}&+\frac{(1-t)v}{tu+(1-t)v}\frac{y}{v}\Big)\\
  &\leq \frac{tu}{tu+(1-t)v}\gamma_{q,r}\Big(\frac{x}{u}\Big)
  +\frac{(1-t)v}{tu+(1-t)v}\gamma_{q,r}\Big(\frac{y}{v}\Big).
}
With the substitution $w:=\frac{x}{u}$, $z:=\frac{y}{v}$ and $\lambda:=\frac{tu}{tu+(1-t)v}$ one can easily see that the above inequality holds for all $x,u,y,v\in(a,b)$ and $t\in[0,1]$ if and only if
\Eq{*}{
  \gamma_{q,r}(\lambda w+(1-\lambda)z)
  \leq \lambda \gamma_{q,r}(w)+(1-\lambda)\gamma_{q,r}(z)
}
is valid for all $w,z\in\big(\frac{a}{b},\frac{b}{a}\big)$ and $\lambda\in[0,1]$, that is, if $\gamma_{q,r}$ is convex over the interval $\big(\frac{a}{b},\frac{b}{a}\big)$.

Thus, we have proved that assertion (i) is equivalent to assertion (ii).

The convexity of $\gamma_{q,r}$ over $\big(\frac{a}{b},\frac{b}{a}\big)$ is valid if and only if $\gamma_{q,r}''(t)\geq 0$ for all $t\in\big[\frac{a}{b},\frac{b}{a}\big]$, i.e., if
\Eq{pq}{
  q(q-1)t^{q-2}\geq r(r-1)t^{r-2}, \qquad t\in\big[\tfrac{a}{b},\tfrac{b}{a}\big].
}
Substituting $t=1$, we get that $(q-r)(q+r-1)\geq0$, which implies that $1\leq q+r$.

Then we have the following four possibilities for the location of $(q,r)$ (keeping in mind that $r<q$).
\Eq{*}{
 (1)\quad 0\leq r\leq 1\leq q;\qquad
 (2)\quad r<q<1\leq q+r;\qquad
 (3)\quad r<0 \mbox{ and } 1\leq q+r;\qquad
 (4)\quad 1<r<q.
}

In the case (1), the inequality \eq{pq} holds for all $t>0$, because the left hand side is nonnegative and the right hand side is nonpositive and conditions is equivalent (iii)(1).

In the case (2), we have that $r,q\in(0,1)$, therefore both sides of the inequality \eq{pq} are negative, and hence it is equivalent to the following inequality
\Eq{*}{
  \beta_{q,r}\leq \frac{1}{t}, \qquad t\in\big[\tfrac{a}{b},\tfrac{b}{a}\big],
}
which turns out to be equivalent to (iii)(2).

In the cases (3), and (4), we can see that both sides of the inequality \eq{pq} are positive and 
it is equivalent to the following inequality
\Eq{*}{
  \beta_{q,r}\geq \frac{1}{t}, \qquad t\in\big[\tfrac{a}{b},\tfrac{b}{a}\big],
}
which turns out to be equivalent to (iii)(3) and (iii)(4), respectively.
\end{proof}


\end{document}